\documentclass[12pt, twoside, leqno]{article}

\usepackage{amsmath,amsthm}
\usepackage{amssymb}

\usepackage{enumitem}

\usepackage{graphicx}

\usepackage[T1]{fontenc}

\newtheorem{theorem}{Theorem}[section]
\newtheorem{corollary}[theorem]{Corollary}
\newtheorem{lemma}[theorem]{Lemma}
\newtheorem{proposition}[theorem]{Proposition}

\theoremstyle{definition}
\newtheorem{definition}[theorem]{Definition}
\newtheorem{remark}[theorem]{Remark}
\newtheorem{example}[theorem]{Example}

\numberwithin{equation}{section}

\begin{document}


\baselineskip=17pt


\title{Invariant $\varphi $-means on multiplier completion of Banach algebras with application to hypergroups}

\author{Mehdi Nemati\\
Department of Mathematical Sciences\\
Isfahan University of Technology\\
Isfahan 84156-83111, Iran\\
E-mail: m.nemati@cc.iut.ac.ir
\and 
Maryam Rajaei Rizi\\
Department of Mathematical Sciences\\
Isfahan University of Technology\\
	Isfahan 84156-83111, Iran\\
E-mail: m.rajaierizi@math.iut.ac.ir}

\date{}

\maketitle


\renewcommand{\thefootnote}{}

\footnote{2010 \emph{Mathematics Subject Classification}: Primary 46H20, 43A62,; Secondary 43A15.}

\footnote{\emph{Key words and phrases}: Fourier algebra, invariant $\varphi$-mean, multiplier norm, ultraspherical hypergroup.}

\renewcommand{\thefootnote}{\arabic{footnote}}
\setcounter{footnote}{0}


\begin{abstract}
Let ${\mathcal A}$ be a Banach algebra and let $\varphi $ be a non-zero character on ${\mathcal A}$. Suppose that ${\mathcal A}_M$ is the closure of the faithful Banach algebra ${\mathcal A}$ in the multiplier norm. In this paper, topologically left invariant $\varphi$-means on ${\mathcal A}_M^*$ are defined and studied. Under
some conditions on ${\mathcal A}$, we will show that the set of topologically left invariant $\varphi$-means on ${\mathcal A}^*$ and on ${\mathcal A}_M^*$ have the same cardinality. We also study the left uniformly continuous functionals associated with these algebras. The main applications are concerned with the Fourier algebra of an ultraspherical hypergroup $H$. In particular, we  obtain some characterizations of discreteness of $H$. 
\end{abstract}

\section{Introduction}

Let $G$ be a locally compact group and let $A_M(G)$ denote the closure of the Fourier algebra
$A(G)$ of $G$  in the multiplier norm. 
Recently, Forrest and Miao  \cite{fm} studied topologically invariant means on $A_M(G)$, and also studied the  uniformly
continuous functionals associated with some important  group algebras such
as $A(G)$ and $A_M(G)$. Using the above theory, the authors proved that there is a bijection between left invariant means on  $VN(G)$ and $A_M(G)^*$. Various links with weakly almost periodic functionals, and with Arens regularity, are also explored. See also \cite{Nem} for quantum group generalizations of these results.

In this paper, we shall prove an extension of these connections in the case of Banach algebra settings.

Let $H$ be an ultraspherical hypergroup associated to a locally compact group
$G$
and a spherical
projector $\pi: C_c(G)\longrightarrow C_c(G)$ which was introduced
and studied in \cite{M}. In recent years people have become interested in studying the properties of the Fourier algebra $A(H)$. For example, Shravan Kumar  showed in \cite{Sh} that there is a unique topological
invariant mean on $A(H)^*$ if and only if $H$ is discrete. More recently, Degenfeld-Schonburg, Kaniuth and Lasser  investigated spectral synthesis properties of Fourier algebras on  ultraspherical hypergroups in \cite{Kan}. Alaghmandan characterized in \cite{Alag2} 
the existence of bounded approximate identities for  Fourier algebras
of ultraspherical hypergroups. 

The paper is organized as follows. In Section 2, some preliminaries and notations are presented
 that will be used throughout the rest of the paper. In Section 3, we study on topologically left invariant $\varphi$-means on ${\mathcal A}_M^{*}$ and prove that for a faithful Banach algebra ${\mathcal A}$ with a certain condition, there is a bijection between topologically left invariant $\varphi$-means with norm one on ${\mathcal A}^*$ and on ${\mathcal A}_M^*$. We also show that the number of topologically left invariant $\varphi$-means with norm one on ${\mathcal A}^*$ is equal to the number of topologically left invariant $\varphi$-means with norm one on ${\rm LUC}({\mathcal A}^*)$. We then prove that ${\mathcal A}_{M}$ is an ideal in its second dual if and only if ${\mathcal A}$ is an ideal in its second dual. 
In Section 4, we give some application and additional results on hypergroup algebras, in particular on Fourier algebras of  ultraspherical hypergroups. For example, we show that there exists a unique topologically  invariant mean on $A_M(H)^{*}$ if and only if $H$ is discrete. Moreover, we prove that $\rm {UCB}(\hat{H})\subseteq \rm{WAP}(A(H)^*)$ if and only if $H$ is discrete. 

\section{Preliminaries}

Let ${\mathcal A}$ be a Banach algebra with norm $\Vert \cdot \Vert$. The Banach algebra ${\mathcal A}$ is called {\it faithful} if for each $a \in {\mathcal A} \setminus \lbrace 0 \rbrace$, there are $ b, c \in {\mathcal A}$ such that $ab \neq 0$ and $ca \neq 0$. We define the {\it multiplier seminorm} $\|\cdot\|_M$ on ${\mathcal A}$ by
$$
\|a\|_M=\sup_{b\in {\mathcal A}, \|b\|\leq 1}\{\|ab\|, \|ba\|\}\quad (a\in {\mathcal A}).
$$
It is easy to see that $\|\cdot\|_M$ is a norm on ${\mathcal A}$ if and only if ${\mathcal A}$ is faithful. Moreover, for each $a\in {\mathcal A}$ we have $\|a\|_M\leq \|a\|$. For faithful Banach algebra ${\mathcal A}$ 
we denote ${\mathcal A}_M$ to be the completion of ${\mathcal A}$ with respect to
the multiplier norm. Then, ${\mathcal A}_M$ is a Banach algebra and contains ${\mathcal A}$ as a norm dense two-sided ideal.

Let $B$ be a Banach algebra with norm $\|\cdot \|_B$. Then the Banach algebra ${\mathcal A}$ is a {\it Segal algebra} in $B$ if ${\mathcal A}$ is a dense, two-
sided ideal in $B$ and there exists a constant $l>0$ such that
$$
\|a\|_B\leq l\|a\|\quad (a\in {\mathcal A})
$$
Therefore, ${\mathcal A}$ is a Segal algebra in ${\mathcal A}_M$. Thus, by \cite[Theorem 2.1]{bu} we obtain that the character spaces of these two Banach algebras, denoted by $\Delta({\mathcal A})$ and $\Delta({\mathcal A}_M)$, are homeomorphic.
It is easy to check that ${\mathcal A}^*$ is a Banach ${\mathcal A}$-bimodule by the following actions:
$$\langle f\cdot a, b\rangle=\langle f, ab\rangle,\quad \langle a\cdot f, b\rangle=\langle f, ba\rangle \quad (f \in {\mathcal A}^*,~ a, b \in {\mathcal A}).$$
We define the space of {\it left uniformly continuous functionals} on ${\mathcal A}$ as follows 
$$\rm{LUC} ({\mathcal A}^*)=\overline{\langle {\mathcal A}^*\cdot {\mathcal A}\rangle}^{\Vert \cdot \Vert_{{\mathcal A}^*}}.$$
Here, for a set $Y$ in a Banach space, $\langle Y\rangle$ denotes the linear span of
$Y$ in the space. It is known that there are two multiplications $\square$ and $\diamondsuit$ on the second dual ${\mathcal A}^{**}$ of ${\mathcal A}$, called, respectivelly, the left and the right Arens products, each extending the multiplication on ${\mathcal A}$. The left Arens product $\square$ is induced by the left ${\mathcal A}$-module structure on ${\mathcal A}$. That is, for $m, n \in {\mathcal A}^{**}$, $f \in {\mathcal A}^*$ and $a \in {\mathcal A}$ we have
$$\langle m \square n, f\rangle=\langle m, n\cdot f\rangle, \quad \langle n\cdot f, a\rangle=\langle n, f\cdot a\rangle.$$
Similarly, one can define the right Arens product.
A Banach right ${\mathcal A}$-submodule $X$ of ${\mathcal A}^*$ is called left introverted if $m\cdot f \in X$ for all $m \in X^*$ and $f \in X$. 
In this case, $X^*$ is a Banach algebra with the multiplication induced by the left Arens product $\square$ inherited from ${\mathcal A}^{**}$.
Examples of left introverted subspace of ${\mathcal A}^*$ are $\rm{LUC}({\mathcal A}^*)$ and $\rm{WAP}({\mathcal A}^*)$, the space of {\it weakly almost periodic functional} on ${\mathcal A}$; that is all $f\in {\mathcal A}^*$ such that the set 
$\lbrace a\cdot f:~a \in {\mathcal A},~ \Vert a \Vert \leq 1 \rbrace$
is relatively weakly compact in ${\mathcal A}^{*}$.

\section{Topologically invariant $\varphi$-means and uniformly continuous functionals}
Throughout this section let ${\mathcal A}$ be a Banach algebra and $\varphi$ a character on ${\mathcal A}$.
We start  with the next definition which is the main object of the section.
\begin{definition}
Let ${\mathcal A}$ be a Banach algebra and let $\varphi\in \Delta({\mathcal A})$ be a character on ${\mathcal A}$. Suppose that $X$ is a Banach right ${\mathcal A}$-submodule of ${\mathcal A}^*$ containing $\varphi$. Then $m \in X^*$ is called a {\it topologically left invariant $\varphi$-mean} (${\rm TLI}_{\varphi}$) on $X$ if $m(\varphi)=1$ and $\langle m, f\cdot a\rangle=\varphi(a)\langle m, f\rangle$ for all $f\in X$ and $a\in {\mathcal A}$. We denote the set of all topologically left invariant $\varphi$-means on $X$ by ${\rm TLI}_{\varphi}(X^*)$. If in addition to the above conditions, $\Vert m \Vert=1$, then $m$ is called a topologically left invariant $\varphi$-mean  with norm one (${\rm TLIM}_{\varphi}$) on $X$. We also denote by ${\rm TLIM}_{\varphi}(X^*)$ the set of all topologically left invariant $\varphi$-means with norm one on $X$. 
\end{definition}
Notice that, the set of  topologically left invariant $\varphi$-means coincides with the set of $\varphi$-means introduced and studied
in \cite{Kan1, Kan2}. A Banach algebra ${\mathcal A}$ is said to be {\it $\varphi$-amenable} if there is a topologically left invariant $\varphi$-mean on ${\mathcal A}^{*}$.
Suppose that ${\mathcal A}$ is a faithful Banach algebra. Then ${\mathcal A}^*$ is a Banach ${\mathcal A}_M$-bimodule with the following module actions: 
$$\langle f\cdot a, b\rangle=\langle f, ab\rangle, \quad \langle a\cdot f, b\rangle=\langle f, ba\rangle,$$
for all $f \in {\mathcal A}^*$, $a \in {\mathcal A}_M$, $b \in {\mathcal A}$. 
Let
$i:~{\mathcal A}\rightarrow {\mathcal A}_M$
be the inclusion map. Then $i$ is an injective Banach ${\mathcal A}_M$-bimodule morphism. Consider the adjoints 
$i^{*}:~{\mathcal A}^{*}_M\rightarrow {\mathcal A}^*$ and $i^{**}:~{\mathcal A}^{**}\rightarrow {\mathcal A}^{**}_M$ of $i$. Since  $i$ has  a dense range, $i^*$ is injective. It is not hard to see that $i^*$ is in fact the restriction map.
Therefore, we can identify ${\mathcal A}^*_M$ with a subset of ${\mathcal A}^*$ and ${\mathcal A}^{**}$ with a subset of ${\mathcal A}^{**}_M$.
\begin{proposition}\label{p:1} 
Let ${\mathcal A}$ be a faithful Banach algebra.
Then the following statements hold.\\
{\rm (i)} ${\mathcal A}^*\cdot {\mathcal A}_M \subseteq {\rm LUC}({\mathcal A}^*)$.\\
{\rm (ii)} $i^*({\rm LUC}({\mathcal A}_M^*)) \subseteq {\rm LUC}({\mathcal A}^*)$.\\
{\rm (iii)} If ${\mathcal A}_M$ has a bounded right approximate identity, then ${\mathcal A}^*\cdot {\mathcal A}_M = {\rm LUC}({\mathcal A}^*)$.\\
{\rm (iv)} $f\cdot a$, $a\cdot f \in i^*({\mathcal A}^{*}_M)$ for all $f \in {\mathcal A}^*$ and $a\in {\mathcal A}$.\\
{\rm (v)} $i^*(a\cdot f)=a\cdot i^*(f)$ and $i^*(f\cdot a)=i^*(f)\cdot a$ for all $a\in {\mathcal A}_M$ and $f \in {\mathcal A}^*_M$.
\end{proposition}
\begin{proof}
{\rm (i)}. If $a \in {\mathcal A}_M$, then there is a sequence $(a_n) \subseteq {\mathcal A}$ such that $\Vert a_n - a\Vert_M \rightarrow 0$. It follows that for each $f \in {\mathcal A}^{*}$ and $b \in {\mathcal A}$, we have
\begin{eqnarray*}
\vert \langle f\cdot a_n - f\cdot a, b \rangle \vert &=&\vert \langle f\cdot (a_n - a), b \rangle \vert\\
&=&\vert \langle f, (a_n - a)b \rangle \vert\\
&\leq & \Vert f \Vert ~~\Vert a_n - a \Vert_M\Vert b \Vert,
\end{eqnarray*}
which implies that
$\Vert f\cdot a_n - f\cdot a\Vert \rightarrow 0$. 
{\rm (ii)}. This follows immediately from (i), since $i^*(f\cdot a) \in {\mathcal A}^*\cdot {\mathcal A}_M$ for all $ f \in {\mathcal A}^*_M$ and $a \in {\mathcal A}_M$.\\
{\rm (iii)}. Since ${\mathcal A}_M$ has a bounded right approximate identity, we obtain from the Cohen's Factorization Theorem and from(i) that ${\mathcal A}^*\cdot {\mathcal A}_M$ is a closed subspace of ${\rm LUC}({\mathcal A}^*)$. On the other hand, ${\mathcal A}^*\cdot {\mathcal A} \subseteq {\mathcal A}^*\cdot {\mathcal A}_M$. Thus, (iii) is true.\\
{\rm (iv)}. Given $f \in {\mathcal A}^*$ and $a\in {\mathcal A}$, define a linear functional on ${\mathcal A}_M$ by $\varphi_{f, a}(b)= \langle f, ab \rangle$ for all $b \in {\mathcal A}_M$. Then we have 
\begin{eqnarray*}
\Vert \varphi_{f, a} \Vert ={\rm sup}_{\Vert b \Vert_M \leq 1} \vert \varphi_{f, a}(b) \vert
\leq \Vert f \Vert\Vert a \Vert< \infty.
\end{eqnarray*}
Moreover, this linear functional agrees with $f\cdot a$ on ${\mathcal A}$ and $f\cdot a=i^*(\varphi_{f, a})$. Similarly, we can show that $a\cdot f \in i^*({\mathcal A}^{*}_M)$.
{\rm (v)}. Let $a\in {\mathcal A}_M$ and $f \in {\mathcal A}^*_M$. Then for each $ b \in {\mathcal A}$, we have
\begin{eqnarray*}
\langle i^*(f\cdot a), b \rangle &=& \langle f\cdot a, i(b) \rangle=\langle f, ai(b) \rangle\\
&=& \langle f, i(ab) \rangle=\langle i^*(f), ab \rangle\\
&=& \langle i^*(f)\cdot a, b \rangle.
\end{eqnarray*}
Thus, $i^*(f\cdot a)=i^*(f)\cdot a$. That $i^*(a\cdot f)=a\cdot i^*(f)$ is similar.
\end{proof}

For each $m\in {\mathcal A}^{**}_M$ define 
\begin{eqnarray*}
m_L : {\mathcal A}^* \rightarrow {\mathcal A}^*,\quad f \mapsto m.f,
\end{eqnarray*}
where the product $m\cdot f \in {\mathcal A}^*$ is given by $\langle m\cdot f, a \rangle=\langle m, f\cdot a \rangle$ for all $a \in {\mathcal A}$. Moreover, it is easy to see that $\Vert m_L \Vert \leq \Vert m \Vert$. For each $\varphi \in \Delta ({\mathcal A})$, we put 
$$I_{\varphi}:=\lbrace a \in {\mathcal A} : \quad \varphi(a)=0\rbrace,$$
the $\varphi$-augmentation ideal in ${\mathcal A}$ which has co-dimension one. 
\begin{theorem}\label{th:1}
Let ${\mathcal A}$ be a faithful Banach algebra with a topologically left invariant $\varphi $-mean of norm one and $\overline{\langle {\mathcal A}\cdot I_{\varphi}\rangle}= I_{\varphi}$. Then $i^{**}({\rm TLIM}_{\varphi}({\mathcal A}^{**}))\subseteq {\rm TLIM}_{\varphi}({\mathcal A}_M^{**})$. Furthermore, the map $i^{**}:{\rm TLIM}_{\varphi}({\mathcal A}^{**})\rightarrow {\rm TLIM}_{\varphi}({\mathcal A}_M^{**})$ is a bijection.
\end{theorem}
\begin{proof}
Let $ m \in {\rm TLIM}_{\varphi}({\mathcal A}^{**})$. Then $i^{**}(m)(\varphi)=1$ and $\|i^{**}(m)\|\leq 1$. Moreover, for each and $a \in {\mathcal A}_M$, there is a sequence $(a_n)$ in ${\mathcal A}$ such that $\Vert a_n - a\Vert_M \rightarrow 0$. Since $\Delta({\mathcal A})=\Delta({\mathcal A}_M)$, we obtain that
$$\varphi(a_n)\rightarrow \varphi(a),\quad \Vert f\cdot a_n - f\cdot a\Vert_{{\mathcal A}^*_M} \rightarrow 0$$ 
for all $f \in {\mathcal A}^*_M$. Therefore, 
\begin{eqnarray*}
\langle i^{**}(m), f\cdot a \rangle &=& \lim_{n \rightarrow \infty}\langle i^{**}(m), f\cdot a_n \rangle\\
&=& \lim_{n \rightarrow \infty}\langle m, i^*(f)\cdot {a_n} \rangle\\
&=& \lim_{n \rightarrow \infty}\varphi(a_n) \langle m, i^*(f)\rangle\\
&=& \varphi(a) \langle i^{**}(m), f \rangle.
\end{eqnarray*}
Hence, $i^{**}(m)\subseteq {\rm TLIM}_{\varphi}({\mathcal A}_M^{**})$ and $i^{**}$ is injective. We now prove that $i^{**}$ is surjective. Suppose that $m \in {\rm TLIM}_{\varphi}({\mathcal A}_M^{**})$. Then for each $ f \in {\mathcal A}^*$ and $a, b \in {\mathcal A}$, we have $f\cdot a \in i^*({\mathcal A}_M^*)$ by Proposition \ref{p:1} (iv), and so 
$$\langle m_L(f), ab \rangle=\langle m, (f\cdot a)\cdot b \rangle=\varphi(b)\langle m_L(f), a \rangle.$$
Thus,
$$\langle m_L(f), ab \rangle=0\quad (f \in {\mathcal A}^*, a\in {\mathcal A}, b \in I_{\varphi}).$$
Since $\overline{\langle {\mathcal A}\cdot I_{\varphi}\rangle}= I_{\varphi}$, it follows that $m_L(f)~|_{I_{\varphi}}=0$ for all $ f \in {\mathcal A}^*$.
As $ab-ba \in I_{\varphi}$, we get that 
$$\langle m, f\cdot (ab) \rangle=\langle m, f\cdot (ba) \rangle \quad (a, b \in {\mathcal A}).$$
Let $n \in {\mathcal A}^{**}$ be such that $\Vert n \Vert =n(\varphi)=1$. Then there is a net $(a_i)$ in ${\mathcal A}$ such that $a_i \overset{w^*}{\longrightarrow}n$ and $\Vert a_i \Vert \rightarrow 1$.
Consider the functional $\tilde{m} \in {\mathcal A}^{**}$ defined by 
$$\tilde{m}(f):= \langle n \square m, f \rangle=\langle n, m.f \rangle \quad (f \in {\mathcal A}^*).$$
By Proposition \ref{p:1} (iv), the functional $\tilde{m}$ is well defined. From the above we conclude that 
\begin{eqnarray*}
\tilde{m}(f\cdot a) &=& \langle n \square m, f\cdot a \rangle=\langle n, m\cdot (f\cdot a) \rangle \\
&=& \langle n, (m\cdot f)\cdot a \rangle=\lim_{i \rightarrow \infty}\langle a_i, (m\cdot f)\cdot a \rangle\\
&=& \lim_{i \rightarrow \infty}\langle m\cdot f, aa_i\rangle=\lim_{i \rightarrow \infty}\langle m\cdot f, a_ia \rangle\\
&=& \lim_{i \rightarrow \infty}\langle m, (f\cdot a_i)\cdot a\rangle=\lim_{i \rightarrow \infty}\varphi(a) \langle m, f\cdot a_i \rangle\\
&=& \varphi(a) \lim_{i \rightarrow \infty} \langle a_i \square m, f \rangle=\varphi(a) \langle n \square m, f \rangle\\ 
&=& \varphi(a) \tilde{m} (f) \quad ( f \in {\mathcal A}^*, a \in {\mathcal A}).
\end{eqnarray*}
Moreover, $\Vert \tilde{m} \Vert \leq 1$ and 
\begin{eqnarray*}
\tilde{m}(\varphi) &=& \langle n \square m,\varphi \rangle=\langle n, m\cdot \varphi \rangle \\
&=& \lim_{i \rightarrow \infty}\langle a_i, m\cdot \varphi \rangle=\lim_{i \rightarrow \infty}\langle m\cdot \varphi, a_i \rangle\\
&=& \lim_{i \rightarrow \infty}\langle m, \varphi\cdot a_i \rangle=\lim_{i \rightarrow \infty}\varphi(a
_i) \langle m, \varphi \rangle\\
&=& m(\varphi) =1.
\end{eqnarray*}
Hence, $\Vert \tilde{m} \Vert = 1$; that means, $\tilde{m} \in {\rm TILM}_{\varphi}({\mathcal A}^{**})$.
\begin{eqnarray*}
\langle i^{**}(\tilde{m}), f \rangle &=& \langle \tilde{m}, i^*(f) \rangle=\langle n \square m, i^*(f) \rangle\\
&=& \langle n, m\cdot i^*(f) \rangle=\lim_{i \rightarrow \infty}\langle a_i, m\cdot i^*(f) \rangle\\
&=& \lim_{i \rightarrow \infty}\langle m\cdot i^*(f), a_i \rangle=\lim_{i \rightarrow \infty} \langle m, i^*(f)\cdot a_i \rangle\\
&=& \lim_{i \rightarrow \infty} \varphi(a_i) \langle m, i^*(f) \rangle =\langle m, f \rangle \quad ( f \in {\mathcal A}_M^*).
\end{eqnarray*}
This shows that $i^{**}(\tilde{m})=m$.
\end{proof}

For each $ n \in {\rm LUC}({\mathcal A}^*)^*$, define a bounded linear map $n_L$ from ${\mathcal A}^*$ into ${\mathcal A}^*$ by 
$ \langle n_L(f), a \rangle= \langle n, f.a \rangle$ for all $f \in {\mathcal A}^*$ and $a \in {\mathcal A}$.
\begin{theorem}\label{th:2}
Let ${\mathcal A}$ be a faithful Banach algebra with a topologically left invariant  $\varphi $-mean of norm one and $\overline{\langle {\mathcal A}\cdot I_{\varphi}\rangle}= I_{\varphi}$. Then the restriction map $R:{\rm TLIM}_{\varphi}({\mathcal A}^{**})\rightarrow {\rm TLIM}_{\varphi}({\rm LUC}({\mathcal A}^*)^*)$ is a bijection.
\end{theorem}
\begin{proof}
It is clear that the restriction map $R$ is well-defined and injective. We now claim that $R$ is surjective.
Let $n \in {\rm TLIM}_{\varphi}({\rm LUC}({\mathcal A}^*)^*)$. Define the functional $\tilde{n} \in {\mathcal A}^{**}$ by 
$$\tilde{n}(f):=\langle m , n_L(f) \rangle \quad (f \in {\mathcal A}^*),$$ where $ m \in {\mathcal A}^{**}$ with $\Vert m \Vert = m(\varphi)=1$. 
Let $(a_i)$ be a bounded net in ${\mathcal A}$ such that $a_i \overset{w^*}{\longrightarrow}m$.
First note that if $f \in {\mathcal A}^*$ and $a, b \in {\mathcal A}$, then 
$$ \langle n_L(f), ab \rangle= \langle n, (f\cdot a)\cdot b \rangle=\varphi(b)\langle n_L(f), a \rangle.$$
Therefore, 
$$ \langle n_L(f), ab \rangle=0 \quad (f \in {\mathcal A}^*, a \in {\mathcal A}, b \in I_{\varphi}).$$
Since we have $\overline{\langle {\mathcal A}\cdot I_{\varphi}\rangle}=I_{\varphi}$, it follows that $n_L(f)|_{I_{\varphi}}=0$ for all $f \in {\mathcal A}^*$. As $ab-ba \in I_{\varphi}$ for all $a, b \in {\mathcal A}$, we get that 
$$ \langle n_L(f), ab \rangle=\langle n_L(f), ba \rangle \quad (a, b \in {\mathcal A}).$$ 
Hence,
\begin{eqnarray*}
\tilde{n}(f\cdot a)&=&\langle m, n_L(f\cdot a) \rangle
= \lim_{i \rightarrow \infty} \langle a_i, n_L(f)\cdot a \rangle\\
&=& \lim_{i \rightarrow \infty}\langle n_L(f), aa_i \rangle =\lim_{i \rightarrow \infty} \langle n_L(f), a_ia \rangle\\
&=& \lim_{i \rightarrow \infty}\langle n, (f\cdot a_i)\cdot a \rangle=\lim_{i \rightarrow \infty} \varphi (a)\langle n, f\cdot a_i \rangle\\
&=& \lim_{i \rightarrow \infty} \varphi (a) \langle a_i, n_L(f) \rangle = \varphi (a)\tilde{n}(f).
\end{eqnarray*}
Furthermore, $\Vert \tilde{n} \Vert \leq 1$ and 
\begin{eqnarray*}
\tilde{n}(\varphi)=m(\varphi)n(\varphi)=1.
\end{eqnarray*}
Thus, $\Vert \tilde{n} \Vert =1$ and $\tilde{n} \in {\rm TLIM}_{\varphi}({\mathcal A}^{**})$.
Finally, if $f \in {\rm LUC}({\mathcal A}^*)$, then 
\begin{eqnarray*}
\tilde{n}(f)&=& \langle m, n_L(f) \rangle
= \lim_{i \rightarrow \infty} \langle n_L(f), a_i \rangle\\
&=& \lim_{i \rightarrow \infty}\langle n, f\cdot a_i \rangle=\lim_{i \rightarrow \infty} \varphi(a_i)\langle n, f \rangle\\
&=& n(f).
\end{eqnarray*}
This shows that $R(\tilde{n})=n$. Hence, $R$ is surjective.
\end{proof}
\begin{lemma}\label{lem:1}
Let ${\mathcal A}$ be a faithful Banach algebra satisfying $\overline{\langle {\mathcal A}\cdot{\mathcal A}\rangle}= {\mathcal A}$. Then $f\cdot a \in {\rm LUC}({\mathcal A}^*_M)$, for all $f \in {\mathcal A}^*$ and $a \in {\mathcal A}$. Moreover, $\Vert f\cdot a \Vert_{{\mathcal A}^*_M} \leq \Vert f \Vert ~~\Vert a \Vert$ and ${\rm LUC}({\mathcal A}^*_M)=\overline{\langle {\mathcal A}^*\cdot {\mathcal A}\rangle}^{\Vert \cdot \Vert_{{\mathcal A}^*_M}}$.
\end{lemma}
\begin{proof}
By Proposition \ref{p:1}(iv), $f\cdot a \in i^*({\mathcal A}^*_M)$ for all $f \in {\mathcal A}^*$ and $a \in {\mathcal A}.$ In addition, we can easily check that $\Vert f\cdot a \Vert_{{\mathcal A}^*_M} \leq \Vert f \Vert  \Vert a \Vert$.
Now, we have to show that $f.a \in {\rm LUC}({\mathcal A}^*_M)$.
Given that $\overline{\langle {\mathcal A}\cdot{\mathcal A}\rangle}= {\mathcal A}$, for each $\epsilon >0$ there exist $a_i, b_i \in {\mathcal A}(i=1, \ldots, n)$ such that
$$\Vert \sum_{i=1}^n a_ib_i-a \Vert < \epsilon.$$
Put $b= \sum_{i=1}^n a_ib_i$. Then $f\cdot b \in {\rm LUC}({\mathcal A}^*_M)$.
Therefore,
$$\Vert f\cdot b-f\cdot a \Vert_{{\mathcal A}^*_M} \leq \Vert f \Vert \Vert b-a \Vert < \epsilon \Vert f \Vert.$$
This shows that $f\cdot a \in {\rm LUC}({\mathcal A}^*_M)$ and we have 
$$\overline{\langle {\mathcal A}^*\cdot {\mathcal A}\rangle}^{\Vert \cdot \Vert_{{\mathcal A}^*_M}}\subseteq {\rm LUC}({\mathcal A}^*_M).$$
To prove the reverse inclusion, for given $a \in {\mathcal A}_M$, there exists a sequence $(a_i) $ in ${\mathcal A}$ for which $\Vert a_i - a \Vert_M \rightarrow 0$. Thus, $\Vert f\cdot a_i - f\cdot a \Vert_{{\mathcal A}^*_M} \rightarrow 0$ for all $ f \in {\mathcal A}^*_M$. It follows that 
\begin{eqnarray*}
{\rm LUC}({\mathcal A}^*_M) &=& \overline{\langle {\mathcal A}_M^*\cdot {\mathcal A}_M\rangle}^{\Vert \cdot \Vert_{{\mathcal A}^*_M}}\\
&=& \overline{\langle {\mathcal A}_M^*\cdot {\mathcal A}\rangle}^{\Vert \cdot \Vert_{{\mathcal A}^*_M}}\\
&\subseteq & \overline{\langle {\mathcal A}^*\cdot {\mathcal A}\rangle}^{\Vert \cdot \Vert_{{\mathcal A}^*_M}}.
\end{eqnarray*}
This implies that
${\rm LUC}({\mathcal A}^*_M) = \overline{\langle {\mathcal A}^*\cdot {\mathcal A}\rangle}^{\Vert \cdot \Vert_{{\mathcal A}^*_M}}.$
\end{proof}
\begin{remark}
Let ${\mathcal A}$ be a Banach algebra satisfying $\overline{\langle {\mathcal A}\cdot I_\varphi\rangle}= I_\varphi$. Then $\overline{\langle {\mathcal A}\cdot {\mathcal A}\rangle}= {\mathcal A}$. In fact, if we choose $a_0\in {\mathcal A}$ such that $\varphi(a_0)=1$, then ${\mathcal A}=I_\varphi\oplus {\Bbb C}a_0$. Thus, there is $0\neq\lambda\in {\Bbb C}$ and $a_1\in I_\varphi$ such that
$a_0^2=a_1+\lambda a_0$. This implies that $a_0\in \overline{\langle {\mathcal A}\cdot {\mathcal A}\rangle}$ and so $\overline{\langle {\mathcal A}\cdot {\mathcal A}\rangle}= {\mathcal A}$.
\end{remark}
\begin{theorem}\label{th:3}
Let ${\mathcal A}$ be a faithful Banach algebra with a topologically left invariant  $\varphi $-mean of norm one satisfying $\overline{\langle {\mathcal A}\cdot I_{\varphi}\rangle}= I_{\varphi}$. Then we have 
$i^{**}({\rm TLIM}_{\varphi}({\rm LUC}({\mathcal A}^*)^*)) \subseteq {\rm TLIM}_{\varphi}({\rm LUC}({\mathcal A}^{*}_M)^*)$. Furthermore, the map 
$i^{**}:{\rm TLIM}_{\varphi}({\rm LUC}({\mathcal A}^*)^*) \rightarrow {\rm TLIM}_{\varphi}({\rm LUC}({\mathcal A}^{*}_M)^*)$ 
is a bijection. 
\end{theorem}
\begin{proof}
By Proposition \ref{p:1}(ii), we have ${\rm LUC}({\mathcal A}^*)^*\subseteq {\rm LUC}({\mathcal A}^{*}_M)^*$. Using Lemma \ref{lem:1} and by simple modifications of the arguments
used in the proof of Theorem \ref{th:1}, we can prove this.
\end{proof}

As an immediate consequence of Theorems \ref{th:1}, \ref{th:2} and \ref{th:3} we obtain the following result on the cardinality of the sets of topologically left invariant $\varphi$-means with norm one. 
\begin{corollary}\label{cor:1}
Let ${\mathcal A}$ be a faithful Banach algebra  satisfying $\overline{\langle {\mathcal A}\cdot I_{\varphi}\rangle}= I_{\varphi}$. Then we have
\begin{eqnarray*}
\vert {\rm TLIM}_{\varphi}({\rm LUC}({\mathcal A} ^*_M)^*)\vert &=& \vert {\rm TLIM}_{\varphi}({\rm LUC}({\mathcal A} ^*)^*)\vert\\
&=& \vert {\rm TLIM}_{\varphi}({\mathcal A} ^{**})\vert\\
&=& \vert {\rm TLIM}_{\varphi}({\mathcal A} ^{**}_M)\vert.
\end{eqnarray*}
\end{corollary}
\begin{remark}
An inspection of the arguments
used in the proof of the above results, one can show that
these results are correct, replacing ${\rm TLIM}_\varphi$ by ${\rm TLI}_\varphi$.
\end{remark} 

Recall that a Banach algebra ${\mathcal A}$ is an ideal in $({\mathcal A}^{**}, \square)$ if and only if the operators 
$l_a : b\mapsto ab$ and
$ r_a : b\mapsto ba$ are both weakly compact for each $a\in {\mathcal A}$; see \cite[Proposition 2.6.25]{Dale}.
\begin{proposition}\label{p3}
Let ${\mathcal A}$ be a faithful Banach algebra with the condition that $\overline{\langle {\mathcal A}\cdot {\mathcal A}\rangle}= {\mathcal A}$. Then ${\mathcal A}_{M}$ is an ideal in its second dual if and only if ${\mathcal A}$ is an ideal in its second dual.
\end{proposition}
\begin{proof}
First suppose that ${\mathcal A}_{M}$ is an ideal in its second dual. Then for every $a\in {\mathcal A}_{M}$, the operator $l_a$, defined as above, is weakly compact on ${\mathcal A}_{M}$. Suppose now that $d=bc$ for some $b,c\in {\mathcal A}$ and let $(a_n)$ be a bounded sequence in ${\mathcal A}$. Then $(a_n)$ is also bounded in ${\mathcal A}_{M}$. Hence, there exists a subsequence $(a_{n_j})$ of $(a_n)$ such that $(ca_{n_j})$ is weakly convergent in ${\mathcal A}_M$. Since ${\mathcal A}^*\cdot {\mathcal A}\subseteq i^*({\mathcal A}_M^*)$, it follows that $(bca_{n_j})$ converges weakly in ${\mathcal A}$. This means $l_d$ is weakly compact on ${\mathcal A}$. By density of 
${\langle {\mathcal A}\cdot {\mathcal A}\rangle}$ in ${\mathcal A}$, we conclude that $l_d$ is weakly compact on ${\mathcal A}$ for all $d\in {\mathcal A}$. Similarly, $r_d$ is weakly compact on ${\mathcal A}$. This shows that ${\mathcal A}$ is an ideal in its second dual.
We next suppose that ${\mathcal A}$ is an ideal in its second dual. Then for every $a\in {\mathcal A}$, $l_a$ is weakly compact on ${\mathcal A}$. Let $d=bc$ for some $b,c\in {\mathcal A}$ and let $(a_n)$ be a bounded sequence in ${\mathcal A}_M$. Then $(c{a_n})$ is a bounded sequence in ${\mathcal A}$ and so there is a subsequence $(c{a_{n_j}})$ of $(c{a_n})$ such that the sequence $(bc{a_{n_j}})$ is weakly convergent in ${\mathcal A}$. Since $i^*({\mathcal A}_M^*)\subseteq {\mathcal A}^*$, it follows that $(da_{n_j})$ is weakly convergent in ${\mathcal A}_M$. By density of $\langle {\mathcal A}\cdot {\mathcal A}\rangle$ in ${\mathcal A}_M$, we obtain that the operator $l_d$ is weakly compact on ${\mathcal A}_M$ for all $d\in {\mathcal A}_M$. Similarly, for every $d\in {\mathcal A}_M$, the operator $r_d$ is also weakly compact on ${\mathcal A}_M$. Thus, ${\mathcal A}_M$ is an ideal in ${\mathcal A}_M^{**}$.
\end{proof}
We must recall some additional concepts. Following \cite{Kan2}, a functional $m\in {\mathcal A}^{**}$ is called a topologically two-sided invariant $\varphi $-mean if $m(\varphi)=1$ and for each $f\in {\mathcal A}^*$ and
$a\in {\mathcal A}$ we have not only $\langle m, f\cdot a\rangle=\varphi(a)\langle m, f\rangle$, but also $\langle m, a\cdot f\rangle=\varphi(a)\langle m, f\rangle$. ${\mathcal A}$ is called two-sided $\varphi $-amenable if there exists a a topologically two-sided invariant $\varphi $-mean on ${\mathcal A}^*$. 
\begin{proposition}\label{p:4}
 Let ${\mathcal A}$ be a Banach algebra. If there is a topologically left invariant $\varphi $-mean $m$ and a topologically right invariant $\varphi $-mean $n$ on $\rm{WAP}({\mathcal A}^*)$, then $m=n$. In particular, if ${\mathcal A}$ is two-sided $\varphi $-amenable, then there is a unique topologically two-sided invariant $\varphi $-mean on $\rm{WAP}({\mathcal A}^*)$.
\end{proposition}
 \begin{proof}
 Suppose that there is a topologically left invariant $\varphi $-mean $m$ and a topologically right invariant $\varphi $-mean $n$ on $\rm{WAP}({\mathcal A}^*)$.  By \cite [Proposition 3.11]{Dale2}, the left and the right Arens products coincide on $\rm{WAP}({\mathcal A}^*)$. Thus, topological invariance implies that
 $$m=n(\varphi )m=n\square m=n\lozenge m=m(\varphi )n=n.$$
 Let $m$ be a topologically two-sided invariant $\varphi$-mean on ${\mathcal A}^{*}$ and suppose that $n$ is another topologically two-sided invariant $\varphi $-mean on $\rm{WAP}({\mathcal A}^*)$. It follows from above that $m=n$ when restricted to $\rm{WAP}({\mathcal A}^*)$.
That is, there is a unique topologically two-sided invariant $\varphi $-mean on $\rm{WAP}({\mathcal A}^*)$. 
\end{proof}
A Banach algebra ${\mathcal A}$ is called weakly sequentially complete if every weakly Cauchy sequence in ${\mathcal A}$
is weakly convergent. It is well-known that the predual of any von Neumann algebra is weakly sequentially complete; see \cite{Ta}. In particular, $L^1(G)$ and ${\mathcal A}(G)$, the group algebra and the Fourier algebra of locally compact group $G$, are weakly sequentially complete.
\begin{corollary}\label{cor:3}
Let ${\mathcal A}$ be a weakly sequentially complete, separable and faithful Banach algebra satisfying $\overline{\langle {\mathcal A}\cdot I_{\varphi}\rangle}= I_{\varphi}$, and suppose that ${\mathcal A}$ is two-sided $\varphi $-amenable. If ${\mathcal A}_M$ is Arens regular, then there is a unique topologically two-sided invariant $\varphi $-mean in ${\mathcal A}$.
\end{corollary}
\begin{proof}
Arens regularity of ${\mathcal A}_M$ implies that $\rm{WAP}({\mathcal A}_M^{*})={\mathcal A}_M^{*}$. By Proposition \ref{p:4} and two-sided version of Corollary \ref{cor:1}, we obtain that there is a unique topologically two-sided invariant $\varphi $-mean on ${\mathcal A}_M^{*}$. Again by two-sided version of Corollary \ref{cor:1}, there is a unique topologically two-sided invariant $\varphi $-mean on ${\mathcal A}^*$. Since $A$ is separable, there is a bounded sequence $(a_n )$ in ${\mathcal A}$ with $\varphi (a_n)=1$ for all $n $ and for each $a\in {\mathcal A}$, 
$$\Vert a{a_n }-\varphi (a){a_n }\Vert \rightarrow 0,\quad\Vert {a_n}a-\varphi (a){a_n }\Vert \rightarrow 0;$$ see \cite [Proposition 3.6]{Kan2}. By uniqueness of topologically two-sided invariant $\varphi $-means on ${\mathcal A}^*$, we conclude that any two $w^*$-cluster points of $(a_n)$ are equal. Since ${\mathcal A}$ is weakly sequentially complete, there exists a $w^*$-cluster point of $(a_n)$ in ${\mathcal A}$ itself, which is the unique topologically two-sided invariant $\varphi $-mean. 
\end{proof} 
\begin{corollary}
Let ${\mathcal A}$ be a weakly sequentially complete and separable Banach algebra and suppose that ${\mathcal A}$ is two-sided $\varphi $-amenable such that ${\rm LUC}({\mathcal A}^*)\subseteq \rm{WAP}({\mathcal A}^{*})$ and $\overline{\langle {\mathcal A}\cdot I_{\varphi}\rangle}= I_{\varphi}$. Then there is a unique topologically two-sided invariant $\varphi $-mean in ${\mathcal A}$.
\end{corollary}
\begin{proof}
Since ${\mathcal A}$ is two-sided $\varphi $-amenable, there is a unique topologically two-sided invariant $\varphi $-mean on $\rm{WAP}({\mathcal A}^*)$, by Proposition \ref{p:4}. Moreover, ${\rm LUC}({\mathcal A}^*)\subseteq \rm{WAP}({\mathcal A}^{*})$. Therefore, it follows from two-sided version of Corollary \ref{cor:1}, that there is a unique topologically two-sided invariant $\varphi $-mean on ${\mathcal A}^*$. Similarly proof of Corollary \ref{cor:3}, we conclude that there is a unique topologically two-sided invariant $\varphi $-mean in ${\mathcal A}$. 
\end{proof}

\section{Applications to hypergroup algebras}

A locally compact Hausdorff space $H$ is a hypergroup if there is a convolution product, denoted by $\ast $, defined on $M(H)$, the space of bounded Radon measures on $H$, with which several general conditions are satisfied. We refer the reader to \cite{Bl} for more details on hypergroups. Suppose that $dx $ is a left Haar measure on $H$ and let $\dot{x}\mapsto \overline{\dot{x}}$ be an involution of $H$. The convolution product on hypergroup algebra $L^1(H)$ is naturally defined to make it a Banach algebra. Define $\varphi _{1}:L^1(H)\rightarrow {\Bbb C}$ by $\varphi _{1}(f)={\int}_H f(\dot{x})~d\dot{x}$ for all $f\in L^1(H)$. It is not hard to show that $\varphi _{1}$ is the identity of the Von Neumann algebra $L^{\infty }(H)$ such that $\varphi _{1} \in \Delta (L^{1}(H))$ and induces $L^1(H)$ to a Lau algebra \cite{L83}. 
Note that a locally compact hypergroup $H$ is amenable if there exists a topologically left invariant $\varphi _{1}$-mean with norm one on $L^{\infty }(H)$ \cite{S}. In what follows, $H$ will always be a hypergroup with a fixed left Haar measure 
$d\dot{x}$. Let $G(H)=\lbrace \dot{x}\in H:\quad \delta _{\dot{x}} \ast \delta _{\overline{\dot{x}}}= \delta _{\overline{\dot{x}}} \ast \delta _{\dot{x}}= \delta _{\dot{e}}\rbrace $. Then $G(H)$ is a closed subhypergroup of $H$ and a locally compact group \cite[10.4C] {J}.
We put $${\rm LUC}(H):={\rm LUC}(L^{\infty}(H)).$$
\begin{example}
Let $H$ be a non compact amenable hypergroup such that the maximal subgroup $G(H)$ is open. Then by \cite[Theorem 5.5] {S} and Corollary \ref{cor:1}, $$\vert {\rm TLIM}_{\varphi _{1}}({\rm LUC}(H)^*) \vert =2^{2^d},$$ where $d$ is the smallest cardinality of a cover of $H$ by compact sets.
\end{example} 
\begin{example}
Let $H$ be an amenable, second countable and noncompact locally compact hypergroup. Then $L^{1}(H)$ is separable. Now, from \cite[Theorem 3.1]{Kan2} and Corollary \ref{cor:1}, we conclude that 
$$\vert {\rm TLIM}_{\varphi_1}(L^\infty (H))|=\vert {\rm TLIM}_{\varphi_1}({\rm LUC}(H)^* \vert =2^{c}.$$
\end{example}
Let $H$ be an spherical hypergroup associated to a locally compact group
$G$
and a spherical
projector $\pi$ which was introduced
and studied in \cite{M}. Then $\pi$ extends to a norm decreasing linear map on various function spaces, including the Fourier algebra $A(G)$; see \cite[Remark 2.2]{M}. A function $f$ is called $\pi$-radial if $\pi (f)=f$, and a measure $\mu$ is called $\pi$-radial if $\pi ^*(\mu)=\mu$.
Note that $H$ is a locally compact hausdorff space equipped with the natural quotient topology under the quotient map $p:G \rightarrow H$. Identifying $M(H)$ with the space of all $\pi$-radial measures in $M(G)$ and restricting the convolution on $M(G)$ to $M(H)$ makes $M(H)$ a Banach algebra. With this convolution structure, $H$ becomes a hypergroup, called a {\it spherical hypergroup} \cite[Theorem 2.12]{M}. A spherical hypergroup is called {\it ultraspherical} if the modular function on $G$ is $\pi$-radial. All double coset hypergroups and
hence all orbit hypergroups are contained
in the class of ultraspherical hypergroups.

 Recall that the character space $\Delta(A(H))$ of $A(H)$ can be canonically identified with $H$. More precisely, the map $\dot{x}\mapsto \varphi_{\dot{x}}$, where $\varphi_{\dot{x}}(u) = u(\dot{x})$ for $u\in A(H)$ is a homeomorphism from $H$ onto $\Delta(A(H))$. The Fourier algebra
$A(H)$ is semisimple, regular and Tauberian\cite[Theorem 3.13]{M}. As in the group case, let $\rho$ also denote the left regular representation of $H$ on $L^2(H)$. This can be extended to $L^1(H)$. Let $C^*_{\rho}(H)$ denote the completion of $\rho(L^1(H))$
in $B(L^2(H))$ which is called the reduced $C^*$-algebra of $H$. The von Neumann algebra generated
by $\{\rho(\dot{x}): \dot{x}\in H\}$ is called the von Neumannn algebra of $H$, and is denoted by $VN(H)$. Note that $VN(H)$ is isometrically isomorphic to the dual of $A(H)$ and by \cite[Theorem 3.5]{Sh1}, the set ${\rm TLIM}_{\varphi_{\dot{e}}}(VN(H))$ is nonempty. Since $A(H)$ is commutative, every topologically left invariant $\varphi_{\dot{e}}$-mean
is automatically a topologically two-sided invariant $\varphi_{\dot{e}}$-mean. 
 For brevity, we use the term topologically invariant mean  rather than topologically left invariant $\varphi_{\dot{e}}$-mean and use the following notations.
\begin{eqnarray*}
A_M(H):=\overline{A(H)}^{\Vert \cdot \Vert_M},\quad
{\rm UCB}(\hat{H})&:=&{\rm LUC}(VN(H)),\\
{\rm UCB}_M(\hat{H})&:=&{\rm LUC}(A_M(H)^*).
\end{eqnarray*}
In what follows, $H$ will always be an ultraspherical hypergroup on  locally compact group $G$. 
\begin{remark}\label{rem3}
Let $u\in I_{\varphi_{\dot{e}}}=\{u\in A(H): u(\dot{e})=0 \}$. Since $\{\dot{e}\}$ is a set of synthesis for $A(H)$ by \cite[Corollary 3.2]{Kan}, we
can suppose that $u$ has compact support. Using the regularity of $A(H)$, we find $v\in A(H)$ such that $v|_{{\rm supp} (u)}\equiv 1$, which implies that $u=uv\in A(H)\cdot I_{\varphi_{\dot{e}}} $. This
proves that $I_{\varphi_{\dot{e}}}=\overline{\langle A(H)\cdot I_{\varphi_{\dot{e}}}\rangle}$.
\end{remark}
\begin{example}
Let $H$ be a second countable locally compact ultraspherical hypergroup. Then $L^2(H)$ is separable and therefore, $A(H)$ is separable. Following \cite[Theorem 4.4]{Sh1}, if $H$ is not discrete, then there isn't a topologically  invariant mean in $A(H)$. Hence, by \cite[Theorem 3.1]{Kan2} and Corollary \ref{cor:1}, we have
$$\vert {\rm TLIM}_{\varphi_{\dot{e}}}({\rm UCB}(\hat{H})^*) \vert =\vert {\rm TLIM}_{\varphi_{\dot{e}}}({\rm UCB}_M(\hat{H})^*) \vert =\vert {\rm TLIM}_{\varphi_{\dot{e}}}(A_M(H)^{**})\vert =2^{c}.$$
\end{example}

\begin{lemma}\label{cor:5}
 Let $H$ be an ultraspherical hypergroup. Then  the following statements are equivalent.
 	
 	{\rm (i)} $H$ is discrete.
 	
 	{\rm (ii)} There is a unique topologically  invariant mean on $A_M(H)^{*}$.
 	
 	{\rm (iii)} $A_M(H)$ is an ideal in its second dual.
 \end{lemma}
\begin{proof} 
 	(i)$\Rightarrow$(ii). Suppose that $H$ is discrete. Then the characteristic  function  $\chi_{\dot{e}}$,  belongs to $A(H)$. However, $\chi_{\dot{e}}$  is  a topologically invariant mean in $A(H)$. By density of $A(H)$  in $A_M(H)$ we conclude that $\chi_{\dot{e}}$ is also a topologically invariant mean in  $A_M(H)$. To show the uniqueness. Suppose  that $m$  is  another topologically  invariant mean on $A_M(H)^{*}$. Then topological invariance and commutativity of $A_M(H)$ implies that $m=\chi_{\dot{e}}\square m=m\square \chi_{\dot{e}}=\chi_{\dot{e}}$.
 	
 	(ii)$\Rightarrow$(i). This follows immediately from Corollary \ref{cor:1} and \cite[Theorem 1.7]{Sh}.
 	
 	(iii)$\Leftrightarrow$(ii). By Remark \ref{rem3} and Proposition \ref{p3} it suffices to show that $H$ is discrete if and only if $A(H)$ is an ideal in its second dual. However, this follows from \cite[Lemma 3.2]{Alag}. 
 \end{proof}	

\begin{corollary}
 Let $H$ be an ultraspherical hypergroup. If $A_M(H)$ is Arens regular, then $H$ is discrete.
 \end{corollary}
 \begin{proof}
 If $A_M(H)$ is Arens regular, then $\rm{WAP}(A_M(H)^{*})=A_M(H)^{*}$. From Proposition \ref{p:4} we conclude that there is a unique topologically  invariant mean on $A_M(H)^{*}$. Hence, by Lemma \ref{cor:5}, $H$ is discrete.
 \end{proof}
\begin{proposition}\label{pr1}
 Let $H$ be an ultraspherical hypergroup. Then $\rm {UCB}(\hat{H})\subseteq \rm{WAP}(A(H)^*)$ if and only if $H$ is discrete.
 \end{proposition} 
 \begin{proof}
 If $\rm {UCB}(\hat{H})\subseteq \rm{WAP}(A(H)^*)$, then by Proposition \ref{p:4}, there is a unique topologically invariant mean on $\rm {UCB}(\hat{H})$. Therefore, by Corollary \ref{cor:1} and Lemma \ref{cor:5}, $H$ is discrete. For the converse, suppose that $H$ is discrete. Then for each $\dot{x}\in H$ the characteristic function $\chi _{\dot{x}}$ is in $A(H)$. Thus, for each $T\in VN(H)$ and $u\in A(H)$, we have 
$$\langle T\cdot \chi _{\dot{x}}, u\rangle=\langle T, \chi _{\dot{x}}u(\dot{x})\rangle=u(\dot{x})\langle T, \chi _{\dot{x}}\rangle .$$
This shows that $T\cdot \chi _{\dot{x}}=\langle T, \chi _{\dot{x}}\rangle \varphi_{\dot{x}}\in \rm{WAP}(A(H)^*)$.
Now, let $u\in A(H)$. Since $A(H)$ is Tauberian, we
can suppose that $u$ has compact and hence finite support. Thus, $u$ 
is a finite linear combination of characteristic functions on one point sets. Therefore, $T\cdot u\in \rm{WAP}(A(H)^*)$. It follows that ${\rm UCB}(\hat{H})\subseteq \rm{WAP}(A(H)^*)$.
 \end{proof}
 \begin{lemma}\label{lem1}
 Let $H$ be an ultraspherical hypergroup. Then $A_M(H)$ is a regular, commutative and
 Tauberian Banach algebra whose character space is canonically
 identified with $H$.
 \end{lemma}
 \begin{proof}
 It is not hard to check that $A_M(H)$ is commutative. Since $A(H)$ is an abstract Segal algebra in ${{\mathcal A}}_M(H)$, it follows from \cite[Theorem 2.1]{bu} that the character space of $A_M(H)$ is homeomorphic to $H$ and $A_M(H)$ is semisimple. Since $A(H)$ is regular, for a closed set $F\subseteq H$ and $\dot{x}\in H\diagdown F$, there is $u\in A(H)\subseteq A_M(H)$ such that $u|_F\equiv 0$ , and $u(\dot{x})=1$. Thus, $A_M(H)$ is regular. Finally, since $A(H)$ is Tauberian and dense in $A_M(H)$, we conclude that $A_M(H)$ is also Tauberian.
 \end{proof}
 The following result follows directly from Lemma \ref{lem1} and arguments
 similar to the proof of Proposition \ref{pr1}.
\begin{corollary}
 Let $H$ be an ultraspherical hypergroup. Then $\rm {UCB}_M(\hat{H})\subseteq \rm{WAP}(A_M(H)^*)$ if and only if $H$ is discrete.
 \end{corollary} 
Following \cite[Lemma 3.7]{Kan}, if $G$ is amenable, then $A(H)$ has an approximate identity $(u_\alpha)$ of norm bound $1$. Hence, any $w^*$-cluster point $E$ of $(u_{\alpha})$ is a right identity for $ (VN(H)^{*}, \square)$ with $\Vert E \Vert =1$. Also, $E \square u= u \square E=u$ for all $u \in A(H)$. Define $\mathcal{E}:=\mathcal{E}(A(H))$ to be the set of all right identities of $VN(H)^*$ that are $w^*$-cluster points of approximate identities in $A(H)$ bounded by one. The following results are generalizing those of  Lau and Losert \cite{llo} obtained for these spaces in the group setting.
\begin{lemma} \label{lem:2}
Let $G$ be an amenable locally compact group and $E$ be a right identity of $VN(H)^*$ with norm one. Then the closed right ideal $E \square VN(H)^*$ of $VN(H)^*$ is isometrically isomorphic to Banach algebra ${\rm UCB}(\hat{H})^*$. 
\end{lemma}
\begin{proof}
Let $(u_{\alpha}) \subseteq A(H)$ with $\Vert u_{\alpha} \Vert \leq 1$ and $u_{\alpha} \overset{w^*}{\longrightarrow} E$ in $VN(H)^*$. It is clear that $E\square VN(H)^*$ is a closed right ideal of $VN(H)^*$ and the restriction map 
$$r : E \square VN(H)^* \rightarrow {\rm UCB}(\hat{H})^*$$ is well-defined, linear and contractive. Now, if $ m \in E \square VN(H)^*$, then there exists $n \in VN(H)^*$ such that $m=E \square n$. Hence, for each $T \in VN(H)$, we have
\begin{eqnarray*}
\langle m, T \rangle &=& \langle E \square n, T \rangle =\langle E \square E\square n, T \rangle \\
&=& \langle E \square m, T \rangle=\lim \langle m, T\cdot u_{\alpha} \rangle \\
&=& \lim \langle r(m), T\cdot u_{\alpha} \rangle.
\end{eqnarray*}
This shows that $\Vert m \Vert = \Vert r (m) \Vert$.
Now, we show that $r$ is surjective. Let $n^{\prime}\in {\rm UCB}(\hat{H})^*$. Then by the Hahn-Banach theorem, there is $n \in VN(H)^*$ such that $n|_{{\rm UCB}(\hat{H})}=n^{\prime}$. If $T \in {\rm UCB}(\hat{H})$, then by the Cohen's Factorization theorem, there exist $u \in A(H)$ and $S \in VN(H)$ such that $T=S\cdot u$. Therefore, 
\begin{eqnarray*}
\langle E \square n, T \rangle &=& \lim \langle n, T\cdot u_{\alpha} \rangle =\lim \langle n^{\prime}, T\cdot u_{\alpha} \rangle \\
&=& \lim \langle n^{\prime}, S\cdot (u{u_{\alpha}}) \rangle= \langle n^{\prime}, S\cdot u \rangle \\
&=& \langle n^{\prime}, T \rangle ;
\end{eqnarray*} 
that is $E \square n \in E \square VN(H)^*$ and $E \square n |_{{\rm UCB}(\hat{H})}=n^{\prime}$
\end{proof}
\begin{proposition}
Let $H$ and $H'$ be two ultraspherical hypergroups. If $A(H)$ has an approximate identity bounded by one and $VN(H)^*$ and $VN(H^{\prime})^*$ are isometrically isomorphic, then ${\rm UCB}(\hat{H})^*$ and ${\rm UCB}(\hat{H^{\prime}})^*$ are also isometrically isomorphic.
\end{proposition} 
\begin{proof}
Suppose that $\gamma :VN(H)^* \rightarrow VN(H^{\prime})^*$ is an isometrically isomorphic. If $E$ is a right identity of $VN(H)^*$ with norm one, then we can easily see that $\gamma (E)$ is also a right identity of $VN(H^{\prime})^*$ with norm one. Hence, by Lemma \ref{lem:2}, we have 
$$ {\rm LUC}(\hat{H})^*\cong E\square VN(H)^* \cong \gamma(E)\square VN(H^{\prime})^* \cong {\rm LUC}(\hat{H^{\prime}})^*.$$
Consequently, ${\rm LUC}(\hat{H})^*$ and ${\rm LUC}(\hat{H^{\prime}})^*$ are isometrically isomorphic.
\end{proof}
Let $G$ be an amenable locally compact group. We put 
$$\Lambda (H)=\bigcap _{E \in \mathcal{E}} E \square VN(H)^*.$$
It is clear that $A(H) \subseteq \Lambda (H)$, since $E \square u=u$ for all $u \in A(H)$ and $E \in \mathcal{E}$.
Also, each $E \square VN(H)^*$ is a closed right ideal in $VN(H)^*$. Thus, $\Lambda(H)$ is a closed right ideal in $VN(H)^*$. 
\begin{proposition}\label{p:5}
Let $G$ be an amenable locally compact group. Then $A(H)$ is an ideal in $\Lambda(H)$ if and only if $H$ is discrete. 
\end{proposition}
\begin{proof}
Suppose that $A(H)$ is an ideal in $\Lambda(H)$ and let $(u_{\alpha})$ be an approximate identity for $A(H)$ with $\Vert u_{\alpha} \Vert \leq 1$. For $m \in VN(H)^*$ and $u \in A(H)$, we have $u \square m \in \Lambda(G)$, since $ E \square (u \square m)=(E \square u) \square m =u \square m$ for all $E \in \mathcal{E}$. Therefore, $u_{\alpha}u \square m \rightarrow u \square m$. Moreover, $u_{\alpha} u \square m =u_{\alpha} \square (u \square m) \in A(H)$, since $A(H)$ is an ideal in $\Lambda(H)$.
Hence, $u \square m \in A(H)$, which means that $A(H)$ is a right ideal in $VN(H)^*$. As $A(H)$ belong to the center of $VN(H)^*$, we conclude that $A(H)$ is an ideal in $VN(H)^*$. Therefore, by  Proposition \ref{p3} and Lemma \ref{cor:5}, $H$ is a discrete hypergroup. Conversely, if $H$ is discrete, again by Proposition \ref{p3} and  Lemma \ref{cor:5}, $A(H)$ is an ideal in $VN(H)^*$ which implies that $A(H)$ is an ideal in $\Lambda(H)$. 
\end{proof}
\begin{proposition}
Let $G$ be an amenable locally compact group and let $E\in VN(H)^*$. Then the following conditions are equivalent.\\
{\rm (i)} $E\in \mathcal{E}$.\\
{\rm (ii)} $E\geq 0$, $u\square E=u$ for all $u\in A(H)$.\\
{\rm (iii)} $\Vert E\Vert=1$ and $E$ is a right identity for $VN(H)^*$.\\
{\rm (iv)} $\Vert E\Vert=1$ and $E$ restricted to ${\rm UCB}(\hat{H})$ is the unit of ${\rm UCB}(\hat{H})^*$. 
\end{proposition}
\begin{proof}
$(i)\Rightarrow (ii)$. Since $E\in \mathcal{E}$, it is clear that $u\square E=u$ for all $u\in A(H)$. Also, if $u\in A(H)$ with $u(\dot{e})=1$, then $$\langle E,\varphi_{\dot{e}}\rangle=\langle E,u\cdot \varphi_{\dot{e}}\rangle=\langle E\square u,\varphi_{\dot{e}}\rangle= u(\dot{e})=1.$$ Since $\Vert E\Vert \leq 1$, it follows that $\Vert E\Vert=\langle E,\varphi_{\dot{e}}\rangle=1$. This shows that $E\geq 0$.

$(ii)\Rightarrow (iii)$. By the same argument as used in the proof of $(i)\Rightarrow (ii)$, we have $\langle E,\varphi_{\dot{e}}\rangle=1$. Furthermore, $E\geq 0$ which implies that $\Vert E\Vert=\langle E,\varphi_{\dot{e}}\rangle=1$. It is trivial to
notice that $E$ is a right identity for $VN(H)^*$.

$(iii)\Rightarrow (i)$. Let $(u_\alpha )$ be a net of states in $A(H)$ such that $u_\alpha \overset{w^*}\longrightarrow E$. Then $(u_\alpha )$ is a weak approximate identity in $A(H)$. A standard argument as used in the proof of \cite[Theorem 1]{Day}, shows that we can find a net $(v_\beta )$ consisting of convex combination of elements in $(u_\alpha )$ such that $\Vert v_\beta u-u\Vert \rightarrow 0$ for all $u\in A(H)$ and $v_\beta \overset{w^*}\longrightarrow E$ in $VN(H)^*$. Thus $E\in \mathcal{E}$. 

$(i)\Rightarrow (iv)$. Since $E\in \mathcal{E}$, there is a bounded approximate identity $(u_\alpha )$ for $A(H)$ such that $u_\alpha \overset{w^*}\longrightarrow E$. Moreover, for each $T\in {\rm UCB}(\hat{H})$, there exist $S\in VN(H)$ and $u\in A(H)$ such that $T=S.u$. Hence, for each $m\in {\rm UCB}(\hat{H})^*$ we have
\begin{eqnarray*}
\langle E\square m, T\rangle &=& \lim \langle u_\alpha \square m,T\rangle =\lim \langle u_\alpha ,m\cdot T\rangle \\
&=& \lim \langle m,T\cdot u_\alpha \rangle = \lim \langle m,(S\cdot u)\cdot u_\alpha \rangle \\
&=& \langle m,S\cdot u \rangle = \langle m,T \rangle.
\end{eqnarray*} 
Therefore, $E\square m=m$. This shows that $E$ is a left identity when restricted to ${\rm UCB}(\hat{H})$. Thus, $E|_{{\rm UCB}(\hat{H})}$ is the unit of ${\rm UCB}(\hat{H})^*$. 

$(iv)\Rightarrow (ii)$. Since $\varphi_{\dot{e}}\in {\rm UCB}(\hat{H})$, by the same argument as used in the proof of $(i)\Rightarrow (ii)$, we obtain that $\langle E,\varphi_{\dot{e}}\rangle=1$. Thus, $E\geq 0$ on $VN(H)$. Let $u\in A(H)$. Since $A(H)$ is Tauberian, we
can suppose that $u$ has compact support. Using the regularity of $A(H)$, there exists $v\in A(H)$ such that $v|_{{\rm supp}(u)}\equiv 1$. Since $T\cdot v\in {\rm UCB}(\hat{H})$ for any $T\in VN(H)$, we obtain from assumption that
\begin{eqnarray*}
\langle u\square E,T\rangle &=& \langle E\square u\square v,T \rangle =\langle E\square u,v\cdot T \rangle \\
&=& \langle E\square u,T\cdot v \rangle = \langle u ,T\cdot v \rangle \\
&=& \langle uv,T\rangle 
=\langle u,T\rangle. 
\end{eqnarray*}
Hence, $u\square E=u$ for all $u\in A(H)$.
\end{proof}


\end{document}